\documentclass[a4paper]{amsart}
\usepackage{hyperref}   
\usepackage{amssymb,xcolor}
\usepackage{stmaryrd} % \owedge
\usepackage{enumerate}
\newcommand{\unfoldedcomment}[7]{}
\newcommand{\foldedcomment}[7]{}

\usepackage{mathtools}  
\mathtoolsset{showonlyrefs}
\newtheorem{theorem}{Theorem}[section]
\newtheorem{corollary}[theorem]{Corollary}
\newtheorem{definition}[theorem]{Definition}
\newtheorem{lemma}[theorem]{Lemma}
\newtheorem{remark}[theorem]{Remark}
\newtheorem{proposition}[theorem]{Proposition}
\numberwithin{equation}{section}
\numberwithin{figure}{section}

\newtheorem{conjecture}[theorem]{Conjecture}

\begin{document}

\title{Rigidity and weak notions of spectral scalar curvature }

\author{Xiaoxiang Chai}
\address{School of Mathematics and Statistics, Key Laboratory of Nonlinear Analysis \& Applications (Ministry of Education), Central China Normal University, Wuhan 430079, P. R. China.}
\email{xxchai@kias.re.kr}
\keywords{Spectral scalar curvature, spectral dihedral rigidity, spectral mean curvature, polyhedra}
\subjclass[2020]{53C24}

\begin{abstract}
  We give a weak formulation of spectral scalar curvature bounded from below
  via establishing a dihedral rigidity result. Given some special convex
  polyhedron, if another metric are of non-negative spectral scalar curvature
  in the interior, with weighted mean-convex faces, and with its dihedral
  angles less than or equal to their flat polyhedral model everywhere along
  the edges, then the metric must be flat. This is motivated by Gromov's
  definition of a weak notion of non-negative scalar curvature. 
\end{abstract}

{\maketitle}

\section{Introduction}

Gromov {\cite{gromov-dirac-2014}} initiated the study of scalar curvature for
\( C^0 \) metrics and he gave the following definition.

\begin{definition}
  \label{c0 nonnegative scalar}Given a continuous metric \( g \) on \( M \),
  we say that \( R_g \geqslant 0 \) in the \( C^0 \) sense at a point \( p \in
  M \) if around \( p \) there does {\underline{not}} exist a cube such that
  its face is strictly mean-convex and the dihedral angles are acute.
\end{definition}

It can be equivalently formulated in the dihedral rigidity conjecture for
cubes {\cite{gromov-dirac-2014}} which states that a Riemannian metric on a
Euclidean convex polyhedron of non-negative scalar curvature, with weakly
mean-convex faces and non-obtuse angles must be flat. There is an alternative
definition using the Ricci flow {\cite{guim-pointwise-2019}}.

The dihedral rigidity conjecture has been proved in several cases: Li
confirmed the conjecture for Euclidean frusta and pyramids
{\cite{li-polyhedron-2020}} in dimension 3 with some additional assumptions as
well as for \( n \)-prisms {\cite{li-dihedral-2024}}, \ up to dimension seven.
Brendle {\cite{brendle-scalar-2024}} and Brendle-Wang
{\cite{brendle-gromovs-arxiv-2023}} confirmed the Euclidean conjecture with
additional hypothesis on angles. The most general case was claimed by
Wang-Xie-Yu {\cite{wang-gromovs-2022-arxiv}}. There is also a hyperbolic
version of the conjecture for parabolic cubes, see Gromov
{\cite{gromov-dirac-2014}}, Li {\cite{li-dihedral-2020-1}}, Wang-Xie
{\cite{wang-dihedral-2023-arxiv}}; for the hyperbolic version modeled on
polyhedra in the upper half-space model of the hyperbolic space, see Chai-Wang
{\cite{chai-dihedral-2024}} and Chai-Wan {\cite{chai-scalar-2024}}. The
dihedral rigidity can be put in a broader context of scalar curvature
rigidity, among which the earliest of such results are the Geroch conjecture
{\cite{schoen-existence-1979}}, {\cite{gromov-positive-1983}} and positive
mass theorems {\cite{schoen-proof-1979}}, {\cite{witten-new-1981}}. In these
early works, two major techniques of scalar curvature geometry, minimal
surface and spinors were developed.

The spectral scalar curvature is defined as the first eigenvalue of an
elliptic operator which is the sum of the Laplacian and the scalar curvature.
We denote the first eigenvalue by \( \lambda_1 (- \gamma \Delta_g +
\tfrac{1}{2} R_g) \). Here, \( R_g \) is the scalar curvature and \( \Delta_g
\) is the Laplacian-Beltrami operator. We only consider the case \( \gamma > 0
\) and the coefficient \( \tfrac{1}{2} \) on \( R_g \) is for convenience. For
a closed manifold \( (M, g) \), \( R_g \geqslant 0 \) obviously implies that
\( \lambda_1 (- \gamma \Delta_g + \tfrac{1}{2} R_g) \geqslant 0 \), hence, \(
\lambda_1 (- \gamma \Delta_g + \tfrac{1}{2} R_g) \geqslant 0 \) is a weaker
condition. Here, we give a slight different version more suitable for
manifolds with boundary.

\begin{definition}
  \label{def spec}Let \( (M, g) \) be a Riemannian manifold and \( u \) be a
  positive function, we call
  \begin{equation}
    - \gamma u^{- 1} \Delta_g u + \tfrac{1}{2} R_g \label{spec sc}
  \end{equation}
  the spectral scalar curvature. Given an oriented hypersurface \( \Sigma \)
  with a chosen unit normal \( N \), we call \( H + \gamma u^{- 1} \partial_N
  u \) the weighted mean curvature. We will explicitly indicate the dependence
  on \( \gamma \) and \( u \) if needed. Here, \( H
  =\ensuremath{\operatorname{div}}_{\Sigma} N \) is the mean curvature of \(
  \Sigma \) in \( (M, g) \).
\end{definition}

In a closed manifold, \( - \gamma u^{- 1} \Delta_g u + \tfrac{1}{2} R_g
\geqslant 0 \), \( u > 0 \) is easily seen to be equivalent to that the first
eigenvalue of the operator \( - \gamma \Delta_g + \tfrac{1}{2} R_g \) is
non-negative.

Analogous to Gromov's definition {\cite{gromov-dirac-2014}} (i.e., Definition
\ref{c0 nonnegative scalar}) of non-negative scalar curvature for \( C^0 \)
metrics, we have the definition: given a continuous positive function \( u \)
and a continuous metric \( g \) on \( M \), we say that \( - \gamma u^{- 1}
\Delta_g u + \tfrac{1}{2} R_g \geqslant 0 \) in the \( C^0 \) sense at a point
\( p \in M \) if around \( p \) there does {\underline{not}} exist a cube such
that its face is strictly weighted mean-convex and the dihedral angles are
acute.

This is just a definition by simply replacing the scalar curvature and the
mean curvature in \ by their spectral, or weighted counterparts. Similar
statements can be made for arbitrary convex polyhedra among which we find it
more convenient to state for the cubes.
{\foldedcomment{+2KXBDVhw1iigpxdB}{+2KXBDVhw1iigpxdC}{comment}{bk21}{1760614511}{}{\

Below is a spectral analog of {\cite[Theorem 1.1]{andersson-rigidity-2008}}.

\begin{theorem}
  \label{thm acg}Let \( 0 \leqslant \gamma < \tfrac{2 n}{n - 1} \), \( \Lambda
  < 0 \) and
  \[ \beta = \frac{\sqrt{- 2 \Lambda}}{\sqrt{2 (n - 1) - (n - 2) \gamma} 
     \sqrt{2 n - (n - 1) \gamma}}, \text{ } \alpha = (2 - \gamma) \beta . \]
  Let \( (M, g) = ([- 1, 1] \times T^{n - 1}, g) \) and \( u > 0 \) such that
  \begin{equation}
    - \gamma u^{- 1} \Delta_g u + \tfrac{1}{2} R_g \geqslant \Lambda,
    \label{neg eigen sc}
  \end{equation}
  and
  \[ H_+ + \gamma u^{- 1} u_{\nu_+} \geqslant (n - 1) \alpha + \gamma \beta
     \text{ along } \partial_+ M =\{1\} \times T^{n - 1}, \]
  and
  \[ H_- + \gamma u^{- 1} u_{\nu_-} \leqslant (n - 1) \alpha + \gamma \beta
     \text{ along } \partial_- M =\{- 1\} \times T^{n - 1} . \]
  Then \( (M, g) \) must be isometric to \( ([t_-, t_+] \times T^{n - 1},
  \mathrm{d} t^2 + e^{2 \alpha t} g_{\mathbb{T}^{n - 1}}) \) with \( u \) for
  some \( t_- < t_+ \) given by a constant multiple of \( e^{\beta t} \).
\end{theorem}}}We would like to formulate \( - \gamma u^{- 1} \Delta_g u +
\tfrac{1}{2} R_g \geqslant \Lambda \) in the \( C^0 \) sense as well, where \(
\Lambda \) is a negative constant and the case \( \gamma = 0 \) was
conjectured by Gromov {\cite{gromov-dirac-2014}} (cf. \ {\cite[Conjecture
1.1]{chai-dihedral-2024}}).

Most generally, we have the following Gromov dihedral rigidity conjecture for
the spectral scalar curvature.

\begin{conjecture}
  \label{spec dihedral}Let \( 0 \leqslant \gamma < \tfrac{2 n}{n - 1} \), \(
  \Lambda \leqslant 0 \) and
  \[ \beta = \frac{\sqrt{- 2 \Lambda}}{\sqrt{2 (n - 1) - (n - 2) \gamma} 
     \sqrt{2 n - (n - 1) \gamma}}, \text{ } \alpha = (2 - \gamma) \beta,
     \text{ } h = (n - 1) \alpha + \beta \gamma . \]
  Let \( \Omega \) be a convex polyhedron in the Euclidean space with a
  distinguished unit vector \( N_0 \), \( N_i \) be the Euclidean unit normal
  vector of the face \( F_i \) of \( \Omega \) pointing outward of \( \Omega
  \). Let \( g \) be a metric and \( u \) be a function defined on a
  neighborhood of \( \bar{\Omega} \). If \( u > 0 \) on \( \bar{\Omega} \) and
  \( g \) satisfies
  \begin{equation}
    - \gamma u^{- 1} \Delta_g u + \tfrac{1}{2} R_g \geqslant \Lambda \text{ in
    } \Omega,
  \end{equation}
  and
  \begin{equation}
    H_{F_i} + \gamma u^{- 1} \tfrac{\partial u}{\partial \nu_i} \geqslant - h
    \langle N_0, N_i \rangle =: - h \cos \bar{\theta}_i \text{ along every }
    F_i,
  \end{equation}
  and the dihedral angles \( \alpha_{i, j} \leqslant \bar{\alpha}_{i, j} \)
  along the edge \( F_i \cap F_j \), then \( (\Omega, g) \) is isometric to
  some polyhedron in the upper half-space model \( \bar{g} =
  \tfrac{1}{\alpha^2 t^2} (\mathrm{d} t^2 + g_{\mathbb{R}^{n - 1}}) \) and \(
  u \) is a constant multiple of \( \bar{u} = t^{- \beta / \alpha} \).
\end{conjecture}

\begin{remark}
  We call \( \Omega \) with the Riemannian metric \( g \) a Riemannian
  polyhedron, and \( \Omega \) with the flat metric a reference polyhedron or
  just a reference which we denote by \( P \).
\end{remark}

\begin{remark}
  \label{horocyclic}In horocyclic coordinates, \( \bar{g} = \mathrm{d} s^2 +
  e^{2 \alpha s} g_{\mathbb{R}^{n - 1}} \) and \( u = e^{\beta s} \) (\( t =
  \alpha^{- 1} e^{- \alpha s} \), \( s = \tfrac{\log (\alpha t)}{- \alpha}
  \)). The metric \( \bar{g} \) and the function \( \bar{u} \) were found by
  the author and Yukai Sun (Henan University, China; \text{{\itshape{in}}}
  \text{{\itshape{preparation}}}) where \( g_{\mathbb{R}^{n - 1}} \) is
  replaced by the flat metric on the torus. 
\end{remark}

\begin{remark}
  There is some freedom to consider the condition
  \begin{equation}
    - \gamma u^{- 1} \Delta_g u + \tfrac{1}{2} R_g + c u^{- 2} | \nabla_g u|^2
    \geqslant \Lambda \label{spec with gradient term}
  \end{equation}
  with suitable range of \( c \) and \( \gamma \), we leave the reader to
  deduce Proposition \ref{rewrite} for \eqref{spec with gradient term} and
  formulate the corresponding Conjecture \ref{spec dihedral}.
\end{remark}

{\foldedcomment{+15kjFFvHKXC7MJh}{+15kjFFvHKXC7MJi}{comment}{bk21}{1761127695}{}{\begin{remark}
  Consider \( \gamma < 2 \) first. It is easy to see that \( \gamma = 2 \), \(
  \alpha = 0 \); a separate discussion might be needed. \( \eta = \beta (2 (n
  - 1) - (n - 2) \gamma) > 0 \).
\end{remark}}}

\subsection{Frustum and pyramid}

With some additional conditions, we confirm Conjecture \ref{spec dihedral} for
two types of polyhedra which we now describe.

\begin{definition}
  \label{frustum}Let \( k \geqslant 3 \) be an integer, \( B_1 \subset \{x^3 =
  0\} \) and \( B_2 \subset \{x^3 = 1\} \) be two similar \( k \)-polygons
  whose corresponding edges are parallel. We call the set \( \{t p + (1 - t) q
  : \text{ } p \in B_1, q \in B_2 \} \) a \( (B_1, B_2) \)-frustum, \( B_1 \)
  its base face and \( B_2 \) its top face. (Most of the time, we just call
  the \( (B_1, B_2) \)-frustum a frustum.)
\end{definition}

A frustum (\text{{\itshape{plural}}}, \text{{\itshape{frusta}}}) is a portion
of a solid that lies between two parallel planes cutting the solid. It is a
solid itself.

\begin{definition}
  \label{pyramid}Let \( k \geqslant 3 \) be an integer, \( B \subset \{x^3 =
  0\} \) be a \( k \)-polygon and \( p \in \{x^3 = 1\} \). We call the set \(
  \{t p + (1 - t) q : \text{ } q \in B\} \) a \( (B, p) \)-pyramid (or just
  pyramid if the references to \( B \) and \( p \) are clear), \( p \) the
  apex of the pyramid and \( B \) the base.
\end{definition}

\begin{figure}[h]
  \resizebox{181pt}{96pt}{\includegraphics{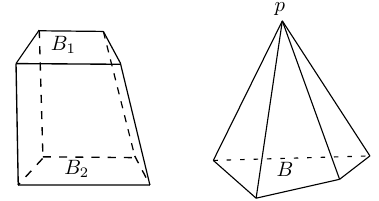}}
  \caption{A frustum and a pyramid.}
\end{figure}

\begin{remark}
  The frustum defined here is what Li {\cite[Definition
  1.3]{li-polyhedron-2020}} called a \( (B_1, B_2) \)-prism and the pyramid is
  what he called a \( (B, p) \)-cone, see also {\cite[Definition
  1.3]{chai-dihedral-2024}}.
\end{remark}

\begin{theorem}
  \label{spec dih fp}Conjecture \ref{spec dihedral} holds in dimension 3 (in
  which case \( 0 \leqslant \gamma < 3 \)) for frustum with the additional
  assumption that the bottom face of the reference is normal to \( N_0 \), and
  for any neighboring faces \( F_j \) and \( F_{j + 1} \), the angle condition
  \begin{equation}
    | \pi - (\alpha_j + \alpha_{j + 1}) | < \alpha_{j, j + 1} \label{technical
    condition}
  \end{equation}
  holds. Here, \( \alpha_j \) is the dihedral angle of \( F_j \) with the base
  face and \( \alpha_{j, j + 1} \) is the dihedral angle formed by \( F_j \)
  and \( F_{j + 1} \). The same conclusion holds for the pyramid if the
  pyramid has an isometric tangent cone at the apex to the tangent cone of its
  Euclidean model or it is a tetrahedron.
\end{theorem}

\begin{remark}
  In dimension 3, \( \beta = \frac{\sqrt{- \Lambda}}{\sqrt{4 - \gamma} \sqrt{3
  - \gamma}} \), \( \alpha = \tfrac{(2 - \gamma) \sqrt{- \Lambda}}{\sqrt{4 -
  \gamma} \sqrt{3 - \gamma}} \) and \( h = \sqrt{\tfrac{4 - \gamma}{3 -
  \gamma}} \). When \( \gamma = 2 \), \( \alpha = 0 \) and it will cause a
  minor issue which can be resolved by changing coordinates, see Remark
  \ref{horocyclic}.
\end{remark}

Theorem \ref{spec dih fp} is the natural spectral analog of
{\cite{li-polyhedron-2020}} and {\cite{chai-dihedral-2024}}, in fact, the
polyhedra considered here are precisely those already considered in
{\cite{li-polyhedron-2020}} and {\cite{chai-dihedral-2024}}, also the
techniques are quite similar. Theorem \ref{spec dih fp} serves as a starting
point for future work on generalizations in higher dimensions and to more
types of polyhedra.

A simple consequence of Theorem \ref{spec dih fp} is that the weak notions of
spectral scalar curvature \( - \gamma u^{- 1} \Delta_g u + \tfrac{1}{2} R_g
\geqslant \Lambda \) for some constant \( \Lambda \leqslant 0 \) makes sense.
Hence, we answered the question of formulation of \( - \gamma u^{- 1} \Delta_g
u + \tfrac{1}{2} R_g \) bounded from below in the \( C^0 \) sense. It might be
interesting to explore this definition using the Ricci flow, cf.
{\cite{guim-pointwise-2019}}, also, it is an interesting question to explore
the preservation of the lower bound \( - \gamma u^{- 1} \Delta_g u +
\tfrac{1}{2} R_g \geqslant \Lambda \) with respect to the connvergence of \(
C^0 \) metrics, see {\cite{gromov-dirac-2014}}. It is also an interesting
question to look for analog of {\cite{chai-scalar-2024}},
{\cite{chai-scalar-2023}} and {\cite{ko-scalar-2024}}, which will be addressed
in a future work.

\

The article is organized as follows:

In Section \ref{sec bubble}, we introduce the capillary warped \( \mu
\)-functional, calculate its first and second variation, in particular, we
relate the spectral curvature condition \eqref{spec sc} to the second
variation. In Sections \ref{sec rig frustum} and \ref{sec rig pyramid}, we
prove the frustum and pyramid case of Theorem \ref{spec dih fp}.

\section{Capillary warped \( \mu \)-bubble}\label{sec bubble}

In this section, we study a capillary version of the warped \( \mu \)-bubble,
in particular, we give the related geometric functional, calculate its first
and second variations. Most importantly, we relate the second variation with
the spectral scalar curvature.

We setup some notations: Let \( \Sigma \) be a surface which meet \( \Omega \)
transversely, and \( E \) be a connected component of \( \Omega \backslash
\Sigma \). Let
\begin{itemize}
  \item \( N \) be the unit normal of \( \Sigma \) in \( \Omega \),
  
  \item \( X \) be the unit outward normal of \( \partial \Omega \) in \(
  \Omega \),
  
  \item \( \nu \) the unit outward normal of \( \partial \Sigma \) in \(
  \Sigma \),
  
  \item \( \eta \) be the unit normal of \( \partial \Sigma \) in \( \Omega \)
  which points outward of \( \partial E \cap \partial \Omega \),
  
  \item and \( \theta \in (0, \pi) \) be the contact angle between \( \Sigma
  \) and \( \partial \Omega \) defined by \( \cos \theta = \langle X, N
  \rangle \).
\end{itemize}
See Figure \ref{fig label}.

\begin{figure}[h]
  \resizebox{181pt}{107pt}{\includegraphics{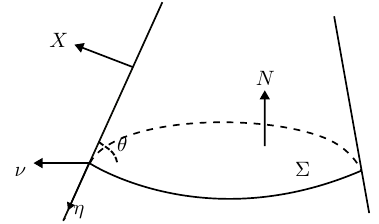}}
  \caption{Labelling of various normal vectors.\label{fig label}}
\end{figure}

\subsection{Capillary warped \( \mu \)-bubble}

We define the warped \( \mu \)-bubble functional
\begin{equation}
  \mathcal{F} (E) = \int_{\partial E \cap \ensuremath{\operatorname{int}}M}
  u^{\gamma} \mathrm{d} \mathcal{H}^2 - \int_E u^{\gamma} \mu \mathrm{d}
  \mathcal{H}^3 - \sum_j \int_{\partial E \cap F_j} u^{\gamma} \cos
  \bar{\theta}_j \mathrm{d} \mathcal{H}^2, \text{ } E \in \mathcal{E},
  \label{energy}
\end{equation}
where \( \mathcal{E} \) is defined as the set of the contractible open sets \(
E' \subset \mathcal{E}' \) and \( \mathcal{E}' \) is given by
\[ \mathcal{E}' := \left\{\begin{array}{l}
     \{E \subset M : p \in E, E \cap B = \emptyset\}, \text{ } P \text{ is a
     pyramid},\\
     \{E \in M : \text{ } B_1 \subset E, \text{ } E \cap B_2 = \emptyset\},
     \text{ } P \text{ is a frustum} .
   \end{array}\right. \]
And \( j \) run through all indices such that \( \{F_j \} \) run through all
side faces. We consider the variational problem
\begin{equation}
  I = \inf \{\mathcal{F}(E) : E \in \mathcal{E}\} . \label{variational
  problem}
\end{equation}
For every \( C^{1, \alpha} \) surface \( \Sigma \) and a smooth family of
diffeomorphisms \( \phi_t : \Sigma \to \Omega \) such that \( \phi_t (\partial
\Sigma) \subset \partial \Omega \), \ \( \Sigma_t = \phi_t (\Sigma) \) such
that \( \Sigma_0 = \Sigma \), we define
\begin{equation}
  \mathcal{A} (t) = \int_{\Sigma_t} u^{\gamma} \mathrm{d} \mathcal{H}^2 +
  \int_{E_t} \mu u^{\gamma} \mathrm{d} \mathcal{H}^3 - \sum_j \int_{\partial
  E_t \cap F_j} \cos \bar{\theta}_j u^{\gamma} \mathrm{d} \mathcal{H}^2,
  \label{Aa}
\end{equation}
where \( E_t \in \mathcal{E} \) is the connected component of \( \Omega
\backslash \Sigma \) closer to the vertex or the top face. The first variation
is given by
\[ \mathcal{A}' (0) = \int_{\Sigma} (\gamma \langle \nabla u, Y \rangle
   u^{\gamma - 1} + u^{\gamma} \ensuremath{\operatorname{div}}_{\Sigma} Y -
   \mu u^{\gamma} \langle Y, N \rangle) - \sum_j \int_{\partial  \Sigma \cap
   F_j} \cos \bar{\theta}_j u^{\gamma} \langle Y, \eta \rangle, \]
here \( Y = \tfrac{\partial}{\partial t} \phi_t \) is the vector field
associated with \( \phi_t \). We decompose \( Y = Y^{\top} + Y^{\bot} \) where
\( Y^{\bot} \) is the component normal to \( \Sigma \). Let \( \phi = \langle
Y, N \rangle \), then \( Y^{\bot} = \phi N \), \( \langle \nabla u, Y \rangle
= \langle \nabla u, Y^{\top} \rangle + \phi \langle \nabla u, N \rangle \), \(
\ensuremath{\operatorname{div}}_{\Sigma} Y = \phi H
+\ensuremath{\operatorname{div}}_{\Sigma} Y^{\top} \) and
\[ \mathcal{A}' (0) = \int_{\Sigma} \ensuremath{\operatorname{div}}_{\Sigma}
   (u^{\gamma} Y^{\top}) + \int_{\Sigma} (H + \gamma u^{- 1} u_N - \mu) \phi -
   \sum_j \int_{\partial \Sigma \cap F_j} \cos \bar{\theta}_j u^{\gamma}
   \langle Y, \eta \rangle . \]
We set \( \cos \theta  = \langle X, N \rangle \), then \( \langle Y, \eta
\rangle = - f / \sin \theta \) and \( \langle Y, \nu \rangle \). Integration
by parts for the first term yields
\begin{align}
\mathcal{A}' (0) & = \int_{\Sigma} (H + \gamma u^{- 1} u_N - \mu) \phi +
\sum_j \int_{\partial \Sigma \cap F_j} u^{\gamma} (\langle Y, \nu \rangle -
\cos \bar{\theta}_j \langle Y, \eta \rangle) \\
& = \int_{\Sigma} (H + \gamma u^{- 1} u_N - \mu) \phi - \sum_j
\int_{\partial \Sigma \cap F_j} \tfrac{u^{\gamma} \phi}{\sin \theta} (\cos
\theta_j - \cos \bar{\theta}_j) . \label{eq first var A}
\end{align}
\begin{definition}
  We say that \( \Sigma \) is a capillary warped \( \mu \)-bubble if \(
  \mathcal{A}' (0) = 0 \) for all \( E_t \). Equivalently,
\begin{align}
H + \gamma u^{- 1} u_N - \mu & = 0 \text{ in } \Sigma, \\
\langle X, N \rangle & = \cos \bar{\theta}_j \text{ along } \partial
\Sigma \cap F_j .
\end{align}
\end{definition}

\begin{definition}
  We say that \( \Sigma \) is a stable capillary warped \( \mu \)-bubble if \(
  \mathcal{A}'' (0) \geqslant 0 \) for all \( E_t \). The inequality \(
  \mathcal{A}'' (0) \geqslant 0 \) is called the stability inequality.
\end{definition}

Now we calculate the second derivatives of \( \mathcal{A} \) if \( \Sigma \)
is a capillary warped \( \mu \)-bubble.

\begin{lemma}
  If \( \Sigma \) is a capillary warped \( \mu \)-bubble, then
  \begin{equation}
    \mathcal{A}'' (0) = \int_{\Sigma} u^{\gamma} \phi (- \Delta_{\Sigma} \phi
    - \gamma u^{- 1} \phi \Delta_{\Sigma} u - \gamma u^{- 1} \langle
    \nabla_{\Sigma} u, \nabla_{\Sigma} \phi \rangle + Z \phi) + \sum_j
    \int_{\partial \Sigma \cap F_j} u^{\gamma} \phi (\tfrac{\partial
    \phi}{\partial \nu} - q \phi), \label{eq second var A}
  \end{equation}
  where \( w = \log u \),
  \begin{equation}
    Z : = - |A|^2 -\ensuremath{\operatorname{Ric}} (N) - \gamma w_N^2 \phi +
    \gamma u^{- 1} \Delta_g u - \gamma H w_N, \label{eq Z}
  \end{equation}
  and
  \begin{equation}
    q = \tfrac{1}{\sin \bar{\theta}_j} A_{\partial M} (\eta, \eta) - \cot
    \bar{\gamma} A (\nu, \nu) . \label{eq q}
  \end{equation}
\end{lemma}

\begin{proof}
  Since \( \Sigma \) is a capillary warped \( \mu \)-bubble, we have
  \begin{equation}
    \mathcal{A}'' (0) = \int_{\Sigma} \delta_Y (H + \gamma u^{- 1} u_N - \mu)
    \phi - \sum_j \int_{\partial \Sigma \cap F_j} \tfrac{u^{\gamma} \phi}{\sin
    \theta_j} \delta_Y (\cos \theta_j - \cos \bar{\theta}_j) . \label{eq first
    second relation}
  \end{equation}
  It suffices to only compute the first variation of \( H + \gamma w_N - \mu
  \) and \( \cos \theta - \cos \bar{\theta} \). First,
\begin{align}
& \delta_Y (H + \gamma w_N - \mu) = \delta_Y (H + \gamma w_N) \\
= & - \Delta_{\Sigma} \phi - (|A|^2 +\ensuremath{\operatorname{Ric}}(N))
\phi - \gamma u^{- 2} u_N^2 \phi + \gamma u^{- 1} \phi \nabla^2_{N N} u -
\gamma u^{- 1} \langle \nabla_{\Sigma} u, \nabla_{\Sigma} \phi \rangle +
\nabla_{Y^{\top}} (H + \gamma w_N) \label{var wmc with tangential}
\\
= & - \Delta_{\Sigma} \phi - \gamma u^{- 1} \phi \Delta_{\Sigma} u -
\gamma u^{- 1} \langle \nabla_{\Sigma} u, \nabla_{\Sigma} \phi \rangle + Z
\phi \label{eq var wmc}
\end{align}
  where we have used \( \nabla^2_{N N} u = \Delta_g u - \Delta_{\Sigma} u - H
  u_N \) and \( Y^{\top} \) is the component of \( Y \) tangential to \(
  \Sigma \). Note that \( \nabla_{Y^{\top}} (H + \gamma w_N) = 0 \) since \( H
  + \gamma w_N \) is constant along \( \Sigma \). Using
  {\cite[Appendix]{ros-stability-1997}}, we find that
  \begin{equation}
    \delta_Y \cos \theta_j = - \sin \theta_j \tfrac{\partial \phi}{\partial
    \nu} + q \phi + \nabla_{Y^{\partial \Sigma}} \cos \theta_j, \label{eq var
    angle}
  \end{equation}
  where \( Y^{\partial \Sigma} \) is the component of \( Y \) tangential to \(
  \partial \Sigma \). Inserting \eqref{eq var wmc} and \eqref{eq var angle}
  into \eqref{eq first second relation} finishes the proof.
\end{proof}

\subsection{Rewrite}Now we give the crucial rewrite which is vital to our
proof.

\begin{proposition}
  \label{rewrite}Let \( \psi = u^{\gamma / 2} \phi \), then
\begin{align}
& \mathcal{A}'' (0) \\
= & \tfrac{4}{4 - \gamma} \int_{\Sigma} | \nabla_{\Sigma} \psi |^2 - (1 -
\tfrac{\gamma}{4}) \gamma \int_{\Sigma} \left| \psi \nabla_{\Sigma} w -
\tfrac{1}{2 (1 - \gamma / 4)} \nabla_{\Sigma} \psi \right|^2 \\
& \quad + \left( \tfrac{1}{2} \int_{\Sigma} R_{\Sigma} \psi^2 + \sum_j
\int_{\partial \Sigma \cap F_j} \kappa_{\partial \Sigma} \psi^2 \right) -
\left( \int_{\Sigma} W \psi^2 + \sum_j \int_{\partial \Sigma}
\tfrac{1}{\sin \bar{\theta}_j} (H_{\partial M} + \gamma \partial_X w - \mu
\cos \bar{\theta}_j) \psi^2 \right), \label{eq rewrite}
\end{align}
  where
  \[ W : = \left( \tfrac{3 - \gamma}{4 - \gamma} \mu^2 + (- \gamma u^{- 1}
     \Delta_g u + \tfrac{1}{2} R_g) \right) + \tfrac{1}{2} (|A|^2 -
     \tfrac{1}{2} H^2) + (1 - \tfrac{1}{4} \gamma) (w_N - \tfrac{1}{4 -
     \gamma} \mu)^2 . \label{eq W} \]
\end{proposition}

\begin{proof}
  It follows from Schoen-Yau's rewrite
  \[ |A|^2 +\ensuremath{\operatorname{Ric}} (N) = \tfrac{1}{2} R_g -
     \tfrac{1}{2} R_{\Sigma} + \tfrac{1}{2} |A|^2 + \tfrac{1}{2} H^2, \]
  of the twicely constracted Gauss equation, \( H = - \gamma w_N + \mu \) and
  suitable regrouping that
  \begin{equation}
    Z = - W + \tfrac{1}{2} R_{\Sigma} \label{eq Z W vanishing} .
  \end{equation}
  By a direct calculation using \( \psi = u^{\gamma / 2} \phi \),
\begin{align}
& u^{\gamma} \phi (- \Delta_{\Sigma} \phi - \gamma u^{- 1} \phi
\Delta_{\Sigma} u - \gamma u^{- 1} \langle \nabla_{\Sigma} u,
\nabla_{\Sigma} \phi \rangle) \\
= & - \psi \Delta_{\Sigma} \psi + (\tfrac{\gamma^2}{4} - \gamma) \psi^2 |
\nabla_{\Sigma} w|^2 - \tfrac{\gamma}{2} \psi^2 \Delta_{\Sigma} w
\\
= & | \nabla_{\Sigma} \psi |^2 + (\tfrac{\gamma^2}{4} - \gamma) \psi^2 |
\nabla_{\Sigma} w|^2 + \gamma \langle \nabla_{\Sigma} w, \nabla_{\Sigma}
\psi \rangle \\
& \quad -\ensuremath{\operatorname{div}}_{\Sigma} (\psi \nabla_{\Sigma}
\psi) -\ensuremath{\operatorname{div}} (\tfrac{\gamma}{2} \psi^2
\nabla_{\Sigma} w) \\
= & \tfrac{4}{4 - \gamma} | \nabla_{\Sigma} \psi |^2 - (1 -
\tfrac{\gamma}{4}) \gamma \left| \psi \nabla_{\Sigma} w - \tfrac{1}{2 (1 -
\gamma / 4)} \nabla_{\Sigma} \psi \right|^2 \\
& \quad -\ensuremath{\operatorname{div}}_{\Sigma} (\psi \nabla_{\Sigma}
\psi) -\ensuremath{\operatorname{div}}_{\Sigma} (\tfrac{\gamma}{2} \psi^2
\nabla_{\Sigma} w) . \label{eq div term}
\end{align}
  Integration of the above omitting the divergence term, and \eqref{eq Z W
  vanishing} show the terms of the integration over \( \Sigma \) in \eqref{eq
  rewrite}.
  
  It is left to deal with the boundary integration. We have collecting the
  divergence term in \eqref{eq div term} and using the divergence theorem that
\begin{align}
& \sum_j \int_{\partial \Sigma \cap F_j} u^{\gamma} \phi (\tfrac{\partial
\phi}{\partial \nu} - q \phi) - \int_{\Sigma}
(\ensuremath{\operatorname{div}}_{\Sigma} (\psi \nabla_{\Sigma} \psi)
+\ensuremath{\operatorname{div}}_{\Sigma} (\tfrac{\gamma}{2} \psi^2
\nabla_{\Sigma} w)) \\
= & \sum_j \int_{\partial \Sigma \cap F_j} (e^{\gamma w / 2} \psi
\tfrac{\partial (e^{- \gamma w / 2} \psi)}{\partial \nu} - \psi
\partial_{\nu} \psi - \tfrac{\gamma}{2} \psi^2 \partial_{\nu} w - q
\psi^2) \\
= & - \sum_j \int_{\partial \Sigma \cap F_j} (\gamma \partial_{\nu} w + q)
\psi^2 .
\end{align}
  Recall that the rewrite (see {\cite[Lemma 3.1]{ros-stability-1997}} or
  {\cite[(4.13)]{li-polyhedron-2020}})
  \begin{equation}
    q = \tfrac{1}{\sin \bar{\theta}_j} A_{\partial M} (\eta, \eta) - \cot
    \bar{\theta}_j A (\nu, \nu) = - H \cot \bar{\theta}_j + \tfrac{H_{\partial
    M}}{\sin \bar{\theta}_j} - \kappa, \label{eq ros}
  \end{equation}
  and using \( H = - \gamma w_N + \mu \), we have
\begin{align}
& \gamma \partial_{\nu} w + q \\
= & \gamma \partial_{\nu} w + (- (\mu - \gamma \partial_N w) \cot
\bar{\theta}_j + \tfrac{H_{\partial M}}{\sin \bar{\theta}_j} - \kappa)
\\
= & - \kappa + \tfrac{1}{\sin \bar{\theta}_j} (H_{\partial M} + \gamma
(\cos \bar{\theta} \partial_N w + \sin \bar{\theta} \partial_{\nu} w) -
\mu \cos \bar{\theta}_j) \\
= & - \kappa + \tfrac{1}{\sin \bar{\theta}_j} (H_{\partial M} + \gamma
\partial_X w - \mu \cos \theta_j) \label{eq boundary rewrite vanish} .
\end{align}
  This finishes the proof.
\end{proof}

\begin{remark}
  When \( \Sigma \) only satisfies the contact angle condition, but \( H +
  \gamma u^{- 1} u_N - \mu =: \tilde{H} \) might not vanish along \( \Sigma
  \), then \( Z \) and \( W \) satisfy
  \begin{equation}
    Z = - W + \tfrac{1}{2} R_{\Sigma} - \tfrac{3}{4} \tilde{H}^2 -
    \tfrac{1}{2} \tilde{H} (3 \mu - \gamma w_N) \label{eq Z W non-vanishing}
  \end{equation}
  instead of \eqref{eq Z W vanishing}. And instead of \eqref{eq boundary
  rewrite vanish},
  \begin{equation}
    \gamma \partial_{\nu} w + q = - \kappa + \tfrac{1}{\sin \bar{\theta}_j}
    (H_{\partial M} + \gamma \partial_X w - \mu \cos \theta_j) - \tilde{H}
    \cot \theta_j . \label{eq boundary rewrite non-vanish}
  \end{equation}
\end{remark}

\section{Rigidity of frustums}\label{sec rig frustum}

In this section, we prove the rigidity of frusta.

\subsection{Infinitesimal rigidity}

\begin{lemma}
  \label{lm local rig}Let \( E \) be a stable capillary \( \mu \)-bubble in \(
  (\Omega, g) \) specified in Theorem \ref{spec dih fp}, then \( \Sigma :=
  \partial E \cap \ensuremath{\operatorname{int}} \Omega \) is infinitesimally
  rigid, that is,
\begin{align}
\nabla_{\Sigma} w & = 0, \text{ } |A|^2 - \tfrac{1}{2} H^2 = 0, \text{ } -
\gamma u^{- 1} \Delta_g u + \tfrac{1}{2} R_g = \Lambda, \\
& \quad w_N = - \tfrac{1}{4 - \gamma} h, \text{ } R_{\Sigma} = 0 \text{
in } \Sigma, \\
H_{\partial M} + \gamma \partial_X w & = - h \cos \bar{\theta}_j, \text{ }
\kappa_{\partial \Sigma} = 0 \text{ along } \partial \Sigma \cap F_j,
\\
\alpha_j & = \bar{\alpha}_j \text{ at } x \in \partial \Sigma \cap E_j .
\end{align}
\end{lemma}

\begin{proof}
  Putting \( \psi = 1 \) in \eqref{eq rewrite} yields
\begin{align}
0 \leqslant & - (1 - \tfrac{\gamma}{4}) \gamma \int_{\Sigma} |
\nabla_{\Sigma} w|^2 \\
& \quad + \left( \tfrac{1}{2} \int_{\Sigma} R_{\Sigma} + \sum_j
\int_{\partial \Sigma \cap F_j} \kappa_{\partial \Sigma} \right) - \left(
\int_{\Sigma} W + \sum_j \int_{\partial \Sigma} \tfrac{1}{\sin
\bar{\theta}_j} (H_{\partial M} + \gamma \partial_X w + h \cos
\bar{\theta}_j) \right) . \label{eq rewrite psi is 1}
\end{align}
  First, by the Gauss-Bonnet theorem,
  \[ \tfrac{1}{2} \int_{\Sigma} R_{\Sigma} + \sum_j \int_{\partial \Sigma \cap
     F_j} \kappa_{\partial \Sigma} + \sum_j (\pi - \alpha_j) = 2 \pi \chi
     (\Sigma), \]
  where \( \alpha_j \) are the interior turning angles of \( \partial \Sigma
  \) at a non-smooth point of \( \partial \Sigma \). Since \( \alpha_j
  \leqslant \bar{\alpha}_j \), and \( \sum_j (\pi - \bar{\alpha}_j) = 2 \pi
  \), so
  \begin{equation}
    \tfrac{1}{2} \int_{\Sigma} R_{\Sigma} + \sum_j \int_{\partial \Sigma \cap
    F_j} \kappa_{\partial \Sigma} \leqslant 0. \label{eq gauss-bonnet angles}
  \end{equation}
  Then we check that \( W \geqslant 0 \) due to \( |A|^2 - H^2 / 2 \geqslant 0
  \), \( - \gamma u^{- 1} \Delta_g u + \tfrac{1}{2} R_g \geqslant \Lambda \)
  and \( \tfrac{3 - \gamma}{4 - \gamma} h^2 + \Lambda = 0 \). Also, \( H +
  \gamma \partial_X w \geqslant - h \cos \bar{\theta}_j \) by the assumptions.
  So the inequality \eqref{eq rewrite psi is 1} is an equality, and tracing
  back, we obtain that
\begin{align}
\nabla_{\Sigma} w & = 0, \\
|A|^2 - \tfrac{1}{2} H^2 & = 0, \\
- \gamma u^{- 1} \Delta_g u + \tfrac{1}{2} R_g & = \Lambda, \\
w_N & = \tfrac{1}{4 - \gamma} h \text{ in } \Sigma, \\
H_{\partial M} + \gamma \partial_X w & = - h \cos \bar{\theta}_j \text{
along } \partial \Sigma \cap F_j, \\
\alpha_j & = \bar{\alpha}_j \text{ at } x \in \partial \Sigma \cap E_j .
\end{align}
  (The second, third and fourth together are implied by \( W = 0 \)). It
  remains to show that \( R_{\Sigma} = 0 \) and \( \kappa_{\partial \Sigma} =
  0 \). First, we see that \( \mathcal{A}'' (0) = 0 \) by the above. Let
\begin{align}
& Q (\psi, \psi) \\
= & \tfrac{4}{4 - \gamma} \int_{\Sigma} | \nabla_{\Sigma} \psi |^2 +
\left( \tfrac{1}{2} \int_{\Sigma} R_{\Sigma} \psi^2 + \sum_j
\int_{\partial \Sigma \cap F_j} \kappa_{\partial \Sigma} \psi^2 \right)
\\
& \quad - \left( \int_{\Sigma} W \psi^2 + \sum_j \int_{\partial \Sigma
\cap F_j} \tfrac{1}{\sin \bar{\theta}_j} (H_{\partial M} + \gamma
\partial_X w - \mu \cos \bar{\theta}_j) \psi^2 \right) .
\end{align}
  Note that \( Q (\psi, \psi) \) differs from the form \eqref{eq rewrite} of
  \( \mathcal{A}'' (0) \) by only one term, and \( Q (\psi, \psi) \geqslant
  \mathcal{A}'' (0) = 0 \). In fact,
  \[ Q (\psi, \psi) = \tfrac{4}{4 - \gamma} \int_{\Sigma} | \nabla_{\Sigma}
     \psi |^2 + \left( \tfrac{1}{2} \int_{\Sigma} R_{\Sigma} \psi^2 + \sum_j
     \int_{\partial \Sigma \cap F_j} \kappa_{\partial \Sigma} \psi^2 \right)
     \geqslant 0. \]
  Let \( \mathcal{L}= - \tfrac{4}{4 - \gamma} \Delta_{\Sigma} + \tfrac{1}{2}
  R_{\Sigma} \), \( \mathcal{B}= \tfrac{4}{4 - \gamma} \partial_{\nu} +
  \kappa_{\partial \Sigma} \). By \eqref{eq gauss-bonnet angles}, \( Q (1, 1)
  = 0 \), so \( \mathcal{L}1 = 0 \) and \( \mathcal{B}1 = 0 \), which gives \(
  R_{\Sigma} = 0 \) and \( \kappa_{\partial \Sigma} = 0 \).
\end{proof}

\subsection{Local foliation}

Now we construct a local foliation near an infinitesimally rigid \( \Sigma \).
We state here a slightly more general version.

\begin{lemma}
  \label{lm exist foliation}Let \( \Sigma \) be a capillary surface of
  prescribed weighted mean curvature , if the linearization of \( H + \gamma
  u^{- 1} u_N + h \) is \( - \Delta_{\Sigma} \) and the linearization of \(
  \cos \theta - \cos \bar{\theta} \) is \( - \sin \theta
  \tfrac{\partial}{\partial \nu} \), then there exists a local foliation \(
  \{\Sigma_t \}_{t \in (- \varepsilon, \varepsilon)} \) near \( \Sigma \) such
  that \( \Sigma_0 = \Sigma \), \( H + \gamma u^{- 1} u_N + h \) is constant
  along \( \Sigma_t \) and \( \theta = \bar{\theta} \) along \( \partial
  \Sigma_t \).
\end{lemma}

\begin{remark}
  A combination of {\cite[Lemma 3.4]{chai-band-2025}} and {\cite[Proposition
  4.1]{li-polyhedron-2020}} finishes the proof, see also
  {\cite{ye-foliation-1991}}, {\cite{ambrozio-rigidity-2015}}.
\end{remark}

Now we derive an ODE for the quantity \( \tilde{H} = H + \gamma u^{- 1} u_N +
h \) along the foliation. This step is the boundary version of {\cite[Lemma
4.4]{chai-band-2025}}. However, {\cite[Lemma 4.4]{chai-band-2025}} is only for
dimension greater than three, here we make use of the Gauss-Bonnet theorem
with boundary and turning angles.

\begin{lemma}
  \label{sign foliation}Let \( \{\Sigma_t \} \) be constructed in Lemma
  \ref{lm exist foliation}, assume that \( (\Omega, g) \) satisfies the
  assumptions of Theorem \ref{spec dihedral}, then
  \begin{equation}
    \tfrac{\mathrm{d}}{\mathrm{d} t} (\exp (- \int_0^t \Psi (s) \mathrm{d} s)
    \tilde{H})' \leqslant 0, \label{integral inequality}
  \end{equation}
  where
  \begin{equation}
    \Psi (t) = (\int_{\Sigma_t} \phi_t^{- 1})^{- 1} \left( - \tfrac{1}{2}
    \int_{\Sigma_t} (3 \mu - \gamma w_N) + \sum_j \int_{\partial \Sigma_t \cap
    F_j} \cot \theta_j \right) . \label{ode psi}
  \end{equation}
\end{lemma}

\begin{proof}
  The first variation \eqref{eq var wmc} gives
  \begin{equation}
    \phi_t^{- 1} \tilde{H}' (t) = - \phi_t^{- 1} \Delta_{\Sigma_t} \phi_t -
    \gamma u^{- 1} \Delta_{\Sigma_t} u - \gamma u^{- 1} \phi_t^{- 1} \langle
    \nabla_{\Sigma_t} u, \nabla_{\Sigma_t} \phi_t \rangle + Z, \label{first
    var in foliation}
  \end{equation}
  which is equivalent to the
  following{\foldedcomment{+1nbB02Ha5l4ByVZ}{+1nbB02Ha5l4ByVa}{comment}{bk21}{1759768625}{}{Setting
  \( \xi_t \) to be \( \phi_t = u^{- \gamma / 2} e^{\xi_t} \), and using
  \eqref{eq Z W non-vanishing}, we see
  \[ \phi_t^{- 1} \tilde{H}' = - | \nabla_{\Sigma_t} \xi_t |^2 -
     \Delta_{\Sigma_t} \xi_t + (\tfrac{\gamma^2}{4} - \gamma) |
     \nabla_{\Sigma_t} w|^2 - \tfrac{\gamma}{2} \Delta_{\Sigma_t} w +
     \tfrac{1}{2} R_{\Sigma_t} - W - \tfrac{3}{4} \tilde{H}^2 - \tfrac{1}{2}
     \tilde{H} (3 h - \gamma w_N) . \]}}
  \[ \phi_t^{- 1} \tilde{H}' (t) = -\ensuremath{\operatorname{div}}_{\Sigma_t}
     \left( \tfrac{\nabla_{\Sigma_t} \phi_t}{\phi_t} + \gamma
     \nabla_{\Sigma_t} w \right) - (1 - \tfrac{\gamma}{4}) \phi_t^{- 2} |
     \nabla_{\Sigma_t} \phi_t |^2 - \gamma \left| \nabla_{\Sigma_t} w +
     \tfrac{\nabla_{\Sigma_t} \phi_t}{2 \phi_t} \right|^2 + Z. \]
  Here, \( \phi_t \) is the variational vector field of the foliation \(
  \{\Sigma_t \} \). We integrate the above on \( \Sigma_t \) and using the
  divergence theorem,
\begin{align}
& \tilde{H}' \int_{\Sigma_t} \phi_t^{- 1} + \int_{\Sigma_t} \left( (1 -
\tfrac{\gamma}{4}) \phi_t^{- 2} | \nabla_{\Sigma_t} \phi_t |^2 + \gamma
\left| \nabla_{\Sigma_t} w + \tfrac{\nabla_{\Sigma_t} \phi_t}{2 \phi_t}
\right|^2 \right) \\
= & - \sum_j \int_{\partial \Sigma_t \cap F_j} (\phi_t \partial_{\nu_t}
\phi_t + \gamma \partial_{\nu} w) + \int_{\Sigma_t} Z \\
= & - \int_{\partial \Sigma_t \cap F_j} (q_t + \gamma \partial_{\nu_t} w)
+ \int_{\Sigma_t} Z \\
= & \sum_j \int_{\partial \Sigma_t \cap F_j} (\kappa_{\partial \Sigma_t} -
\tfrac{1}{\sin \bar{\theta}_j} (H_{\partial M} + \gamma \partial_X w - \mu
\cos \theta_j)) + \tilde{H} \sum_j \int_{\partial \Sigma_t \cap F_j} \cot
\theta_j \\
& \quad + \int_{\Sigma_t} \left( - W + \tfrac{1}{2} R_{\Sigma} -
\tfrac{3}{4} \tilde{H}^2 - \tfrac{1}{2} \tilde{H} (3 \mu - \gamma w_N)
\right),
\end{align}
  where we have used \eqref{eq Z W non-vanishing} and \eqref{eq boundary
  rewrite non-vanish}. Since \( W \geqslant 0 \),
  \begin{equation}
    \tilde{H}' \int_{\Sigma_t} \phi_t^{- 1} \leqslant \tilde{H} \left( -
    \tfrac{1}{2} \int_{\Sigma_t} (3 \mu - \gamma w_N) + \sum_j \int_{\partial
    \Sigma_t \cap F_j} \cot \theta_j \right) . \label{ode}
  \end{equation}
  Solving this ODE, we finish the proof.
\end{proof}

\subsection{Proof of rigidity of frustum}

Now we are ready to prove the dihedral rigidity conjecture for frusta.

\begin{proof}[Proof of Theorem \ref{spec dih fp} for frusta]
  Using the existence and regularity theory in {\cite[Theorem
  2.1]{li-polyhedron-2020}}, there exists a minimiser \( E \) to
  \eqref{variational problem} such that \( \Sigma
  =\ensuremath{\operatorname{int}}M \cap \partial E \) is \( C^{1, \alpha} \)
  up to the corner (Li's theorem is based on scaling argument, and \( u \)
  will play no role in the limit.).
  
  Given any \( \Sigma \), we can define \( \mathcal{A} \) as in \eqref{Aa}.
  Let \( F (t) =\mathcal{A} (\Sigma_t) \) where \( \Sigma_t \) is the leaf of
  the foliation in Lemma \ref{lm exist foliation}. Then by the first variation
  \eqref{eq first var A},
  \[ F' (t) = \int_{\Sigma_t} (H + \gamma u^{- 1} u_N + h) \mathrm{d}
     \mathcal{H}^{n - 2}, \]
  note that there is no boundary term because that the contact angle is \(
  \theta_j = \bar{\theta}_j \) along the edges. Using Lemma \ref{sign
  foliation}, \( F' (t) \leqslant 0 \) for \( t \geqslant 0 \) and \( F' (t)
  \geqslant 0 \) for \( t \leqslant 0 \) which means that every \( \Sigma_t \)
  also gives rise to a minimiser to the functional. By Lemma \ref{lm local
  rig}, every \( \Sigma_t \) is infinitesimal rigid. Now we calculate the
  metric of \( (\Omega, g) \) and \( u \) using the infinitesimal rigidity.
  Moreover, by tracing back the equality, we have that \( \phi_t \) is
  constant. Using \( w_N = - h / (4 - \gamma) \) and \( H = - \gamma w_N - h
  \), we see \( H = 2 (\gamma - 2) h / (4 - \gamma) = - 2 \alpha \) which is
  constant. We now show that \( Y^{\bot} \) is conformal. First, \(
  \nabla_{\partial_i} N = H \partial_i \). Since \( \langle Y, N \rangle \) is
  constant,
\begin{align}
0 & = \nabla_{\partial_i} \langle Y, N \rangle \\
& = \langle \nabla_{\partial_i} Y, N \rangle + \langle Y,
\nabla_{\partial_i} N \rangle \\
& = Y \langle \partial_i, N \rangle - \langle \nabla_Y N, \partial_i
\rangle + H \langle Y, \partial_i \rangle \\
& = - \langle \nabla_Y N, \partial_i \rangle + H \langle Y, \partial_i
\rangle .
\end{align}
  Observe that
  \[ \nabla_Y N = \nabla_{Y^{\top}} N + \nabla_{Y^{\bot}} N = H Y^{\top} +
     \nabla_{Y^{\bot}} N, \text{ } \langle Y, \partial_i \rangle = \langle
     Y^{\top}, \partial_i \rangle, \]
  hence \( \langle \nabla_{Y^{\bot}} N, \partial_i \rangle = 0 \). Moreover,
  \( \nabla_{Y^{\bot}} \langle \partial_i, \partial_j \rangle = \langle
  Y^{\bot}, N \rangle g_{i j} = \tfrac{1}{2} H \phi_t g_{i j} \) by the
  umbilicity \( |A|^2 - \tfrac{1}{2} H^2 = 0 \). Note that every leaf is flat,
  therefore, the local foliation forms a subset \( \cup_t \Sigma_t \) of the
  hyperbolic 3-space with constant curvature \( - | \alpha | \). It follows
  from \( w_N = - h / (4 - \gamma) \) that \( u = t^{- \beta / \alpha} \) (up
  to a constant). We now calculate the second fundamental form of \( \partial
  \Omega \). Let \( e \) be a unit tangent vector of \( \partial \Sigma \).
  Along the face \( F_j \), using the decomposition \( X_j = \cos \theta_j N +
  \sin \theta_j \nu \),
  \[ A_{F_j} (e, e) = \langle \nabla_e X, e \rangle = \cos \theta_j A (e, e) +
     \sin \theta_j \kappa_{\partial \Sigma} = \tfrac{1}{2} H \cos \theta_j .
  \]
  It follows from \( H_{\partial M} + \gamma \partial_X w = - h \cos
  \bar{\theta}_j \) that \( H_{\partial \Omega} = H \cos {\bar{\theta}_j}  \),
  and hence
  \[ A_{F_j} (\eta, \eta) = \langle \nabla_{\eta} X, \eta \rangle =
     H_{\partial \Omega} - \langle \nabla_e X, e \rangle = \tfrac{1}{2} H \cos
     \theta_j . \]
  Note that the vector \( N - \langle \eta, N \rangle \eta \) is of length \(
  \sin \theta_j \), and the direction is the same with \( X \), so
  \[ A_{F_j} (e, \eta) = \langle \nabla_e X, \eta \rangle = \tfrac{1}{\sin
     \theta_j} \langle \nabla_e (N - \langle \eta, N \rangle \eta), \eta
     \rangle = 0. \]
  Hence, every face \( F_i \) is umbilic with curvature \( \tfrac{1}{2} H \cos
  \theta_j \). In the upper half-space model of the hyperbolic 3-space, the
  face is either a part of a sphere or a plane. That each face \( F_i \)
  intersects the leaf \( \Sigma_t \) in a constant angle indicates that it can
  only be a part of a plane. Therefore, by connectedness, we can extend the
  rigidity to all \( \Omega \) and \( (\Omega, g) \) is a polyhedron in the
  upper half-space model.
\end{proof}

\section{Rigidity of pyramids}\label{sec rig pyramid}

In this section, we give the proof for rigidity of pyramids. Our method is to
construct a local foliation near the apex which serves as a barrier for the
existence of the minimisers to the warped \( \mu \)-bubble functional.

First, we construct a local foliation near the apex \( O \).

\begin{proposition}
  \label{foliation unrescaled}Let \( (\Omega, g) \) and \( (\Omega, \delta) \)
  have isometric tangent cones at \( O \). Then there exists a neighborhood \(
  U \) \ of \( O \) foliated by a family of surfaces \( \{\Sigma_t \}_{t \in
  (- \varepsilon, 0)} \) such that each \( \Sigma_t \) is of constant \( H +
  \gamma \omega_N + h \) and meets the face \( F_i \) at the constant angle \(
  \theta_i \).
\end{proposition}

\begin{remark}
  From now on, in some situations, we omit the dependence of \( \theta \) and
  \( \bar{\theta} \) on the indices of the faces for brevity.
\end{remark}

It is more useful to reformulate. The pyramid \( (\Omega, \delta) \) is formed
by truncating its tangent cone at the apex through the base. We let \(
\Sigma_1 \) be the cross-section parallel to the base and of unit distance to
the apex \( O \) and \( \Omega_1 \) to be the pyramid truncated by \( \Sigma_1
\). Let \( \Sigma_t = t \Sigma_1 \) and \( \Omega_t = t \Omega_1 \), \( t > 0
\). Let \( x \in \Omega_1 \), we define \( v (x) = u (t x) \) and \(
(\hat{g}^t)_{i j} (x) = g_{i j} (t x)  \).

We consider
\begin{equation}
  \Sigma_{t, \phi} := \{(\check{x}, - t + \phi (\check{x})) : \text{ }
  (\check{x}, - 1) \in \Sigma_1 \}, \label{perturbation}
\end{equation}
Let every geometric quantity of \( \Sigma_1 \) be denoted with a hat and a
subscript \( t \) with respect to the metric \( \hat{g} := \hat{g}^t \), and
let every geometric quantity on \( \Sigma_{1, \phi} \) be denoted by a hat and
a subscript \( t, \phi \). For example, the unit normal of \( \Sigma_1 \) in
\( \Omega_1 \) with respect to the metric \( \hat{g} \) is given by \(
\hat{N}_t \), and the unit normal of \( \Sigma_{1, \phi} \) is given by \(
\hat{N}_{t, \phi} \).

By rescaling back using Proposition \ref{foliation scaled}, we obtain the
proof of Proposition \ref{foliation unrescaled}. Indeed, let
\[ \hat{\lambda}_{t, \phi} := \hat{H}_{t, \phi} + \gamma v^{- 1}
   \partial_{\hat{N}_{t, \phi}} v + t h. \]
\begin{proposition}
  \label{foliation scaled}There exists a family of functions \( \{\phi (\cdot,
  t)\}_{t \in [0, \varepsilon)} \) defined on \( \Sigma_1 \) such that the
  perturbations \ \( \Sigma_{1, t \phi (\cdot, t)} \) has constant \(
  \lambda_{t, t \phi (\cdot, t)} \) and have constant angles with \( F_i \)
  with respect to the metric \( g^t \).
\end{proposition}

\begin{proof}
  Let \( s \) be a small parameter and the family \( \Sigma_{1, s \phi} \)
  give rise to a vector field \( \partial_s := (0, \phi (\check{x})) \). The
  perturbation \( \Sigma_{1, s \phi} \) of \( \Sigma_1 \) is approximately a
  normal graph over \( \Sigma_1 \) with the graph function \( \hat{\phi} \)
  satisfying
  \[ \hat{\phi} : = s \hat{g} (\partial_s, \hat{N}_t) + O (s^2), \label{approx
     graph} \]
  Setting \( s = t \), then the graph function
  \begin{equation}
    \hat{\phi} = t \phi + O (t^2) \label{approx graph t}
  \end{equation}
  since \( \hat{g} \) converges to the flat metric.
  
  By the first variation of \( \lambda_{t, s \phi} \) and the Taylor expansion
  (with respect to \( s \)),
  \[ \lambda_{t, s \phi} - \lambda_{t, 0} = \hat{L}_s \hat{\phi} + s \langle
     (\partial_s)^{\top}, \nabla ^{\hat{g}_t} \lambda_{t, 0} \rangle + O (s^2)
     = s \hat{L}_s \phi + s \langle (\partial_s)^{\top}, \nabla ^{\hat{g}_t}
     \lambda_{t, 0} \rangle O (s^2) \]
  by \eqref{approx graph} where \( \hat{L}_s \) is define for \( \Sigma_1 \)
  as \eqref{eq var wmc} with respect to the metric \( g^t \). Setting \( s = t
  \) yields
  \[ \lambda_{t, s \phi} = \lambda_{t, 0} + t \hat{L}_t \phi + t \langle
     (\partial_s)^{\top}, \nabla ^{\hat{g}_t} \lambda_{t, 0} \rangle + O (t^2)
     . \]
  By convergence of \( \hat{g} \) to the flat metric and \( u \) converges to
  the flat metric, \( \hat{L}_t = - \Delta_{\Sigma_1} + O (t) \) where \(
  \Delta_{\Sigma_1} \) is the Laplace-Beltrami operator with respect to the
  flat metric (i.e., limit of \( \hat{g} \)) on \( \Sigma_1 \), and \( \langle
  (\partial_s)^{\top}, \nabla ^{\hat{g}_t} \lambda_{t, 0} \rangle = O (t) \).
  So
  \begin{equation}
    \lambda_{t, t \phi} = \lambda_{t, 0} - \Delta_{\Sigma_1} \phi + O (t^2),
    \label{graph approx mean curvature}
  \end{equation}
  Similarly using \eqref{eq var angle},
  \begin{equation}
    \cos \hat{\theta}_{t, t \phi} = \cos \hat{\theta}_{t, 0} - t \sin
    \hat{\theta}_{t, 0} \tfrac{\partial \phi}{\partial \nu_t} + t \hat{q}_t
    \phi + t \langle (\partial_s)^{\partial \Sigma_1}, \nabla \hat{\theta}_{t,
    0} \rangle + O (t^2), \label{graph approximate angle}
  \end{equation}
  where \( \hat{q}_t \) is defined in \eqref{eq q} for \( \Sigma_1 \) with
  respect to the metric \( g^t \).
  
  Since the \( (\Omega_1, g^t) \) converges to the flat pyramid and \( u \)
  converges to a constant, \( \hat{L}_t = - \Delta_{\Sigma_1} + O (t) \) and
  similarly, \( \sin \hat{\theta}  \tfrac{\partial}{\partial \nu_t} = \sin
  \theta  \tfrac{\partial}{\partial \nu_1} + O (t) \) and \( q_t = O (t) \).
  
  Define
  \[ \Psi (t, \phi) = \left( \tfrac{1}{t} \lambda_{t, t \phi} - \tfrac{1}{|
     \Sigma_1 |} \int_{\Sigma_1} \tfrac{1}{t} \lambda_{t, t \phi}, \frac{1}{t
     \sin \theta} (\cos \hat{\theta}_{t, t \phi} - \cos \bar{\theta}) \right),
  \]
  which can be extended to \( t = 0 \) by taking limits \( \Psi (0, \phi) =
  \lim_{t \to 0} \Psi (t, u) \). By the expansion \eqref{graph approx mean
  curvature} and \eqref{graph approximate angle},
  \[ \Psi (0, \phi) = (- \Delta_{\Sigma_1} \phi + \tfrac{1}{| \Sigma_1 |}
     \int_{\Sigma_1} \Delta_{\Sigma_1} \phi, - \tfrac{\partial \phi}{\partial
     \nu_1} + \zeta), \]
  where \( \zeta := \lim_{t \to 0} \tfrac{\cos \hat{\theta}_{t, 0} - \cos
  \bar{\theta}}{t \sin \bar{\theta}} \) is a function on \( \partial \Sigma_1
  \).
  
  By minimising the functional
  \[ I (\phi) = \int_{\Sigma_1} | \nabla_{\Sigma_1} \phi |^2 + \int_{\Sigma_1}
     \zeta \phi \]
  on the space
  \[ \Lambda_0 = \{\phi \in C^{2, \alpha} (\Sigma_1) \cap C^{1, \alpha}
     (\bar{\Sigma}_1) : \text{ } \int_{\Sigma_1} \phi = 0\}, \]
  we can find a solution to the elliptic problem \( \Psi (0, \phi) = 0 \), and
  we set the solution to be \( \phi_0 \).
  
  Now we compute
\begin{align}
D \Psi |_{(0, \phi_0)} (0, \phi) = & \tfrac{\mathrm{d}}{\mathrm{d} s} |_{s
= 0} \Psi (0, \phi_0 + s \phi) \\
= & \tfrac{\mathrm{d}}{\mathrm{d} s} |_{s = 0} s (- \Delta_{\Sigma_1} \phi
+ \tfrac{1}{| \Sigma_1 |} \int_{\Sigma_1} \Delta_{\Sigma_1} \phi, -
\tfrac{\partial \phi}{\partial \nu}) \\
= & (- \Delta_{\Sigma_1} \phi + \tfrac{1}{| \Sigma_1 |} \int_{\Sigma_1}
\Delta_{\Sigma_1} \phi, - \tfrac{\partial \phi}{\partial \nu}),
\end{align}
  since \( \phi_0 \) satisfies \( \Psi (0, \phi_0) = 0 \). Now we apply the
  implicit function theorem. For some sufficiently small \( \varepsilon > 0
  \), there exists a function \( \phi (\cdot, t) \in B (0, \delta) \subset
  \mathcal{X} \), \( t \in (0, \varepsilon) \) such that \( \phi (\cdot, 0) =
  \phi_0 \) and
  \[ \Psi (t, \phi (\cdot, t)) = \Psi (0, \phi_0) = (0, 0) \]
  for every \( t \in [0, \varepsilon) \). In geometric terms, the surface \(
  \Sigma_{t, \phi (\cdot, t)} \) are of constant \( \lambda_{t, \phi (\cdot,
  t)} \) with constant contact angles \( \bar{\theta}_j \) with the face \(
  F_j \).
\end{proof}

\begin{lemma}
  \label{average deficit}Let \( \Sigma_{1, t \phi (\cdot, t)} \) be
  constructed as in Proposition \ref{foliation scaled}, then
  \[ \lambda_{t, t \phi} | \Sigma_1 | = \int_{\Sigma_1} \lambda_{t, 0} +
     \int_{\partial \Sigma_1} \tfrac{1}{\sin \bar{\theta}} (\cos \bar{\theta}
     - \cos \theta ) + O (t^2) . \]
\end{lemma}

\begin{proof}
  This follows from \eqref{graph approx mean curvature}, \eqref{graph
  approximate angle} by integration over \( \Sigma_1 \) and an application of
  the divergence theorem.
\end{proof}

We give a variational formula which gives a relation of the variations of \( -
\gamma u^{- 1} \Delta_g u + \tfrac{1}{2} R_g \), \( H + \gamma u^{- 1}
\partial_N u \) and the dihedral angles.

\begin{proposition}
  Let \( \{u_t \}  \) be a family of positive \( C^2 \) functions and \( \{g_t
  \} \) be a family of smooth metrics on \( \Omega_1 \) converging
  respectively to the constant 1 and the flat metric as \( t \to 0 \). Then
\begin{align}
& [- \int_{\Sigma_1} (H_g + \gamma u^{- 1} \partial_N u) + \int_{\partial
\Sigma} \tfrac{1}{\sin \bar{\theta}} (\cos \bar{\theta} - \cos \theta)]
\\
= & \int_{\Omega_1} (- \gamma u^{- 1} \Delta_g u + \tfrac{1}{2} R_g) +
\int_{\partial \Omega_1 \backslash \Sigma_1} (H_g + \gamma u^{- 1}
\partial_X u) + O (t^2) .
\end{align}
\end{proposition}

\begin{proof}
  The case \( u_t \) is a constant for all \( t \) is due to
  {\cite{miao-mass-2022}}, so we only have to show
  \[ - \int_{\partial \Omega_1} u^{- 1} \partial_N u = \int_{\Omega_1} u^{- 1}
     \Delta_g u + \int_{\partial \Omega_1 \backslash \Sigma_1} u^{- 1}
     \partial_{X_i} u + O (t^2) . \]
  The above follows from that \( u^{- 1} = 1 + O (t) \) and the divergence
  theorem.
\end{proof}

By taking the difference between \( (u_1, g_1) \) and \( (u_2, g_2) \), we
obtain the following.

\begin{corollary}
  \label{sign foliation diff}Let \( \{u_t^{(i)} \}_{i = 1, 2} \) be two
  families of positive \( C^2 \) functions and \( \{g_t^{(i)} \}_{i = 1, 2} \)
  be two families of smooth metrics on \( \Omega_1 \) converging respectively
  to the constant 1 and the flat metric as \( t \to 0 \). Then
\begin{align}
& [- \int_{\Sigma_1} ((H_{g_1} + \gamma u^{- 1} \partial_N u) - (H_{g_2}
+ \gamma u^{- 1} \partial_N u)) + \int_{\partial \Omega_1} \tfrac{1}{\sin
\bar{\theta}} (\cos \theta^{(1)} - \cos \theta^{(2)})] \\
= & \int_{\Omega_1} ((- \gamma u^{- 1} \Delta_g u + \tfrac{1}{2} R_g) - (-
\gamma u^{- 1} \Delta_g u + \tfrac{1}{2} R_g)) \\
& \quad + \int_{\partial \Omega_1 \backslash \Sigma_1} ((H_g + \gamma
u^{- 1} \partial_X u) - (H_g + \gamma u^{- 1} \partial_X u)) + O (t^2) .
\end{align}
\end{corollary}

Now we are ready to finish the proof of Theorem \ref{spec dihedral} for
pyramids.

\begin{proof}[Proof of Theorem \ref{spec dih fp} for pyramids]
  Note that
  \[ \lambda_{t, t \phi} = H_{1, t \phi} + \gamma u^{- 1} \partial_N u + t h,
  \]
  where \( t h \) is \( - (H_{1, t \phi} + \gamma u^{- 1} \partial_N u) \)
  computed with respect to the model. Hence, \( \lambda_{t, t \phi} \) is the
  difference of the weighted mean curvatures with respect to two different
  metrics. Note that \( \bar{\theta} \) is the same with the model. Hence
  subsequent applications of Lemma \ref{average deficit} and Corollary
  \ref{sign foliation diff} shows that \( \lambda_{t, t \phi} \geqslant O
  (t^2) \). Equivalently, by rescaling back, we obtain that \( \tilde{H}_t
  \geqslant O (t) \) for the foliation \( \{\Sigma_t \} \) in Proposition
  \ref{foliation unrescaled}. The condition \( \tilde{H}_t \geqslant O (t) \)
  gives an initial value for the ordinary differential inequality \eqref{ode}
  (with a reversed direction) which is easily seen to hold for \( \{\Sigma_t
  \} \) as well. We write the ODE here
  \[ \tilde{H}' \geqslant \tilde{H} \Psi (t), \Psi (t) = (\tfrac{1}{2}
     \int_{\Sigma_t} (3 h + \gamma w_N) - \sum_j \int_{\partial \Sigma_t \cap
     F_j} \cot \theta_j) (\int_{\Sigma_t} \phi_t^{- 1})^{- 1} \label{ode
     pyramid} \]
  We see that \( \phi_t = 1 + O (t) \), \( \sum_j \int_{\partial \Sigma_t \cap
  F_j} \cot \theta_j = C t + O (t^2) \) for some constant \( C > 0 \) (see
  Remark \ref{proj area}) since the foliation in Proposition \ref{foliation
  unrescaled} is constructed from higher order perturbations of coordinate
  bases. Hence \( \Psi (t) = C t^{- 1} + C_1 (t) \) where \( C_1 (t) \) is a
  bounded continuous function. Hence,
  \[ \frac{\mathrm{d}}{\mathrm{d} t} (\tilde{H} t^C \exp (\int^t C_1 (s)
     \mathrm{d} s)) \geqslant 0 \]
  and we obtain that \( \tilde{H} \geqslant 0 \) for every leaf. This gives a
  barrier for the existence of minimiser to the capillary warped \( \mu
  \)-functional for the polyhedron \( E_t \) obtained by chopping off the
  pyramid below \( \Sigma_t \). Applying the proof for the rigidity of
  frustums, we know that the rigidity holds for \( E_t \), which by taking a
  limit \( t \to 0 \), we obtain the rigidity for the pyramids under the
  assumption of isometric tangent cones at the apex.
  
  As for the case of tetrahedra, if the tangent cones at the apex are not
  isometric, then using {\cite[Proposition 3.14]{chai-dihedral-2024}}, we can
  construct a barrier to the existence of the minimiser \eqref{Aa} near the
  apex. Hence, the tangent cones at the tangent cones must be isometric. Using
  the above proof, finishes the proof of the case of tetrahedra.
\end{proof}

\begin{remark}
  \label{proj area}We explain why \( \sum_j \int_{\partial \Sigma_t \cap F_j}
  \cot \theta_j = C t + O (t^2) \) holds. It suffices to consider the flat
  metric and \( t = 1 \) by rescaling. Let \( O' \) be the projection of \( O
  \) to the plane where \( \Sigma_1 \) lies and \( E_j = F_j \cap \partial
  \Sigma_1 \). We see then that \( \cot \theta_j \) is the signed distance of
  \( O' \) to the line where \( E_j \) lies formed by \( E_j \). So \( |
  \partial \Sigma_1 \cap F_j | \cot \theta_j \) is twice the area of an
  oriented triangle formed by \( O' \) and \( E_j \), and summing over all \(
  F_j \) gives twice the area of \( \Sigma_1 \), that is, the constant \( C
  \).
\end{remark}

\end{document}